\documentclass{amsart}
\usepackage{fix-cm}
\usepackage[T1]{fontenc}

\usepackage{tikz}
\usetikzlibrary{decorations.pathreplacing}

\usepackage{float,url}

\newcommand{\phicir}[1]{\draw[<->,dashed] (#11)  to [out=30,in=150] node[above,draw=none]{$\phi$} (#12);}

\newcommand{\Fi}{\mathit{Fi}_{22}}
\newcommand{\schl}{Schl\"{a}fli symbol}
\newcommand{\orst}{O_8^+(2){:}S_3}
\newcommand{\indinfi}{2^9{\cdot}3{\cdot} 5^2{\cdot} 7}
\newcommand{\olp}{$2^2{\cdot}3^2{\cdot}5{\cdot}13$}

\DeclareMathOperator{\lcm}{lcm}

\theoremstyle{definition}
\newtheorem{theorem}           {Theorem}
\newtheorem{prop}[theorem]    {Proposition}
\newtheorem{coro+}[theorem]            {Corollary}
\newtheorem{lemma}[theorem]            {Lemma}

\newtheorem{defi+}           {Definition}
\newtheorem{problem}         {Problem}
\newtheorem{remark}         {Remark}

\date{ \today}
\begin{document}
\title{Locally toroidal polytopes of rank 6 and sporadic groups}
\author[D.V.~Pasechnik]{Dmitrii V. Pasechnik}
\address{Department of Computer Science and Pembroke College, University of Oxford}
\email{dimpase@cs.ox.ac.uk}

\begin{abstract}
We augment the list of finite universal
locally toroidal regular polytopes of type $\{3,3,4,3,3\}$ due to P.~McMullen and E.~Schulte,
adding as well as removing entries. This disproves a related long-standing conjecture.
Our new universal polytope is related to a well-known
$Y$-shaped presentation for the sporadic simple group $\Fi$, and admits $S_{4}\times\orst$ as the
automorphism group.
We also discuss further extensions of its quotients in the context of $Y$-shaped presentations.
As well, we note that two known examples of finite universal polytopes of type
$\{3,3,4,3,3\}$ are related to $Y$-shaped presentations of orthogonal groups over $\mathbb{F}_2$.
Mixing construction is used in a number of places to describe covers and 2-covers.
\end{abstract}

\maketitle

\section{Introduction and results}
Presentations of finite sporadic simple groups as quotients of Coxeter groups
with diagram $Y_{\alpha\beta\gamma}$, cf. Fig. \ref{fig:Yabc}(a),
were discovered 30 years ago \cite{Atl,CNS88,Soi:Y555}
and remain a subject of considerable interest, cf. e.g.
\cite{Soi:Y555more, CoPr:92, Iv:Y:99, MR2301231, MR2543618, 2014arXiv1403.2401A}.
\begin{figure}[h]
\centering
\begin{tikzpicture}[every node/.style={circle,draw,inner sep=2pt}]
\node [draw] (a)  at (4,5)          {};
\node [draw] (b3) at (4,4)          {} edge (a);
\node [draw] (c3) at (4,3) {} edge[dotted] (b3);
\node [draw] (b2) at (4.866025,5.5) {} edge   (a);
\node [draw] (c2) at (5.73205,6)     {} edge[dotted]   (b2);
\node [draw] (b1) at (3.133974,5.5)  {} edge (a);
\node [draw] (c1) at (2.267949,6)    {} edge[dotted] (b1);
\draw [decorate,decoration={brace,raise=3pt,transform={scale=1.5,xshift=-4pt}}] (c2)--(b2)
     node[below,midway,draw=none]{\footnotesize $\ \ \beta\geq 0$};
\draw [decorate,decoration={brace,raise=3pt,transform={scale=1.5,xshift=-4pt}}] (b1)--(c1)
     node[below,midway,draw=none]{\footnotesize $\alpha\geq 0\ \ $};
\draw [decorate,decoration={brace,raise=3pt,transform={scale=1.5,xshift=-4pt}}] (c3)--(b3)
     node[left,midway,draw=none]{\footnotesize $\gamma\geq 0\mbox{ nodes }\ \ $};
\node [draw=none] (Y) at (4,5.5) [label=$Y_{\alpha\beta\gamma}$] {};
\node [draw=none] (subf) at (2,2.5) [label=(a)] {};
\end{tikzpicture}\qquad
\begin{tikzpicture}[every node/.style={circle,draw,inner sep=2pt}]
\node [draw] (a)  at (4,5)          [label=320:$a$] {};
\node [draw] (b3) at (4,4)          [label=0:$b_3$] {} edge (a);
\node [draw] (c3) at (4,3)          [label=0:$c_3$] {} edge (b3);
\node [draw] (b2) at (4.866025,5.5) [label=$b_2$] {} edge   (a);
\node [draw] (c2) at (5.73205,6)    [label=$c_2$] {} edge   (b2);
\node [draw] (d2) at (6.59808,6.5)    [label=$d_2$] {} edge (c2);
\node [draw] (b1) at (3.133974,5.5)  [label=$b_1$] {} edge (a);
\node [draw] (c1) at (2.267949,6)    [label=$c_1$] {} edge (b1);
\node [draw] (d1) at (1.401924,6.5)    [label=$d_1$] {} edge (c1);
\phicir{b}
\phicir{c}
\phicir{d}
\node [draw=none] (subf) at (2,2.5) [label=(b)] {};
\end{tikzpicture}
\caption{$Y$-diagrams; the general case and $Y_{332}$.\label{fig:Yabc}}
\end{figure}
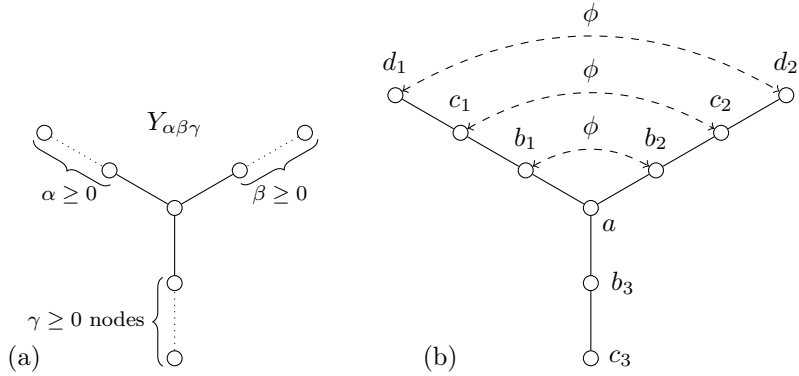
For instance, the sporadic simple group $\Fi$ is a quotient of $Y_{332}$,
cf. Fig.~\ref{fig:Yabc}(b) modulo relations 
\eqref{eq:firel}  below, cf. \cite[p.233]{Atl}.
As explained in Section~\ref{sec:Y} below,
these relations afford the automorphism $\phi$ of the diagram swapping $b_1$,
$c_1$, and $d_1$ with, respectively $b_2$, $c_2$, and $d_2$. Thus
$\Fi$ has a subgroup isomorphic to a quotient of the Coxeter group
with \schl\ $[3,3,4,3,3]$, i.e. with the string diagram we will also denote
by $[3^2,4,3^2]$, cf.  Fig.~\ref{fig:Fabc}.
\begin{figure}[h]
\begin{tikzpicture}[every node/.style={circle,draw,inner sep=2pt}]
\node [draw] (d) {};
\node [draw] (c) [right of=d] {} edge[dotted] (d);
\node [draw] (b) [right of=c] {} edge (c);
\node [draw] (a) [right of=b] {} edge (b);
\node [draw] (u) [right of=a] {} edge[double,thick]
(a);
\node [draw] (s) [right of=u] {} edge (u);
\node [draw] (v) [right of=s] {} edge (s);
\node [draw] (w) [right of=v] {} edge[dotted] (v);
\draw [decorate,decoration={brace,raise=3pt,transform={scale=1.5,xshift=-4pt}}] (c)--(d)
     node[below,midway,draw=none]{\footnotesize $\nu\geq 0$};
\draw [decorate,decoration={brace,raise=3pt,transform={scale=1.5,xshift=-4pt}}] (w)--(v)
     node[below,midway,draw=none]{\footnotesize $\mu\geq 0$};
\end{tikzpicture}
\caption{$[3^{1+\nu},4,3^{1+\mu}]$-diagram. \label{fig:Fabc}}
\end{figure}
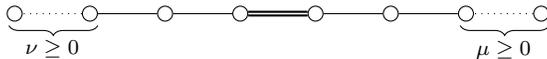

Groups of this kind arise in the study of \emph{abstract regular polytopes}, for which
the book \cite{MS02} by P.~McMullen and E.~Schulte is a definite reference. These objects
may be viewed as quotients (satisfying \emph{intersection property} \eqref{eq:ip}, cf. \cite{MS02})
of Coxeter complexes of Coxeter groups with $r$-node string diagram;
here $r$ is called the \emph{rank}. Abstract regular polytopes are classified
according to topology of their vertex figures and facets; they are called \emph{locally $X$} if
the latter have topology of $X$.
An up-to-date review on group-theoretic approaches to the
problem of classification of abstract regular polytopes may be found e.g. in \cite{MR3427645}.

Following a problem posed by B.~Gr\"{u}nbaum~\cite[p.196]{Gr78}, particular
attention has been paid to finite {\em locally toroidal} abstract regular
polytopes, in this case the rank is bounded from above by $6$, cf. \cite[Lemma~10A.1]{MS02}.

One of the three rank 6 cases is the
case corresponding to the group $[3^2,4,3^2]$, cf. \cite[Sect. 12D]{MS02} (or \cite[Sect. 7]{MS96}),
where a conjecturally complete list of the finite universal examples is given, see
\cite[Table 12D1]{MS02} (which is already in \cite[Table~V]{MS96}) and \cite[Problem~18]{SW06}.
The main results of the present paper give one more example, missing in that table, and
remove erroneous infinite series of examples, of which only first terms
actually exist.

In more detail, the facets $\{3,3,4,3\}_\mathbf{s}$ and vertex figures $\{3,4,3,3\}_\mathbf{t}$
here correspond to nontrivial finite quotients
$[3,3,4,3]_\mathbf{s}$ and $[3,4,3,3]_\mathbf{t}$
of the affine Coxeter group $\tilde{F}_4$, i.e. the groups $[3,3,4,3]$ and $[3,4,3,3]$, with normal
Abelian subgroups either of the form $q^4$ or $q^2\times (2q)^2$, with $q\geq 2$,
where we follow notation from \cite{Atl}
to denote the direct product $(\mathbb{Z}/s\mathbb{Z})^k$ of $k$ copies of the cyclic group of
order $s$ by $s^k$. The case $q^4$ is denoted in \cite{MS96,MS02} by $\{3,3,4,3\}_{(q,0,0,0)}$,
and the case $q^2\times (2q)^2$ by $\{3,3,4,3\}_{(q,q,0,0)}$ (and completely
similarly for $\{3,4,3,3\}$).

\begin{theorem}\label{th:mainmain}
Let $\Gamma$ be a universal locally toroid rank $6$ abstract regular polytope
with vertex figures of type $\{3,4,3,3\}_{(3000)}$ and
facets of type $\{3,3,4,3\}_{(2200)}$.
Then $\Gamma$ is the 24-fold cover of $\Gamma_1$, where $\Gamma_1$ has
$v=11200=2^6{\cdot} 5^2{\cdot} 7$
vertices and $f=14175=3^4{\cdot} 5^2{\cdot} 7$ facets, and the group $\Omega:=\orst$, of
order $g=2^{13}{\cdot} 3^6{\cdot} 5^2{\cdot} 7$.
The group of $\Gamma$ is isomorphic to $S_4\times \Omega$.
\end{theorem}
Here $\Gamma=\Gamma_4$ and $\Gamma_1$ are the biggest and the smallest
member of the sequence of covers $\Gamma_k$, with
$1\leq k\leq 4$. Namely, $\Gamma_k$ is a 
$k!$-fold cover of $\Gamma_1$
and the group of $\Gamma_k$ is isomorphic to $S_k\times \Omega$.
More details on this are given in Section~\ref{sec:Y}. 

A number of comments on Theorem~\ref{th:mainmain} are in order.
The claim of finiteness of $\Gamma$
depends upon coset enumeration, that was, for robustness purposes,
carried out using two different implementations of the
Todd-Coxeter algorithm: the built-in implementation of {\sf GAP} system \cite{GAP4}, and
{\sf ACE} implementation by G.~Havas and C.~Ramsay \cite{ace}, also available as a {\sf GAP}
package \cite{acegap}.
Both computations in the case $\mathbf{t}={(3000)}$, $\mathbf{s}=(2200)$
returned the index of the vertex stabiliser, the subgroup $F:=3^4{:}F_4$
in the quotient $\Theta$ of the Coxeter group $[3^2,4,3^2]$ modulo
the relations on Fig.~\ref{fig:eeF4}
equal to $268800=\indinfi$. We explain below that on the other hand
the identification of the group of
$\Gamma_1$ does not need coset enumeration, and can be carried out by hand or
with any software capable of multiplying $24\times 24$ matrices over $\mathbb{F}_2$, the
field of 2 elements. It is worth mentioning that both $F$ and the facet stabiliser
$(2^2.4^2){:}F_4$
are maximal subgroups in $\Omega$, although only $F$ remains maximal upon
restricting to the simple subgroup $O_8^+(2)$ of $\Omega$.

In Section~\ref{sec:Y} we sketch how to establish the existence
of the group of $\Gamma_1$ as a subgroup in $\Fi$.
Apart from this, we discuss there other connections to the $Y$-shaped presentations already
mentioned, and pose a number of related open problems.
In Section~\ref{sec:gamma} we give another proof of
existence of $\Gamma_1$ and its covers $\Gamma_k$ for $2\leq k\leq 4$ by a computer-free argument.
The latter is an example of an application of Proposition~\ref{prop:prod} below.

\medskip

$\Fi$ is not the only quotient of the Coxeter group $Y_{332}$ related to polytopes
of type $\{3,3,4,3,3\}$. Namely, one also finds groups for such polytopes within the 2-modular
quotient $G$ of $Y_{332}$, which is isomorphic to $2^8{:}O_8^-(2){:}2$, as established by R.~Griess
in \cite{MR704292}, where a general study of $p$-modular quotients of $Y_{\alpha\beta\gamma}$
is carried out; in case $p=2$ one speaks of subgroups of
$\mathit{GL}_{1+\alpha+\beta+\gamma}(2)$ generated by \emph{transvections}---images of the
generators of $Y_{\alpha\beta\gamma}$.
One such polytope is described in the following result, and the other one is
its 2-quotient. The group $G$ also appears on \cite[pp.232-233]{Atl},
as a quotient of $X_{332}$, a re-labelling of $Y_{332}$.
See Section~\ref{sec:2000} for details.

\begin{theorem}\label{th:maintwo}
Let $\Gamma$ be a locally toroid rank $6$ abstract regular polytope
with vertex figures of type $\{3,4,3,3\}$ and
facets of type $\{3,3,4,3\}_{(2000)}$.
Then its vertex figures are either of type $\{3,4,3,3\}_{(2200)}$,
and $\Gamma$ has $128$ facets, $32$ vertices, and the group of order $2^{18}{\cdot} 3^2$,
or its vertex figures are of type $\{3,4,3,3\}_{(2000)}$, and 
$\Gamma$ has $32$ facets, $32$ vertices, and the group of order $2^{15}{\cdot} 3^2$.
\end{theorem}

Basically, enumeration
of the cosets of a facet stabiliser
in the quotient of the Coxeter group $[3^2,4,3^2]$ by a relation forcing
the facets $\{3,3,4,3\}_{(2000)}$ shows that this group
is finite, of order $2^{18}{\cdot} 3^2$. This \emph{proves}
that the vertex figures are of type $\{3,4,3,3\}_{(2200)}$ (unless
further relations are imposed).
This identifies $\Gamma$ with
the second entry in \cite[Table 12D1]{MS02} (or \cite[Table~V]{MS96}) for $t=2$,
i.e.  $\mathbf{t}=(2200)$.

The quotient by a normal subgroup $2^3$ produces
the other case,  $\mathbf{t}=(2000)$, the first entry in [loc.cit.] for $t=2$.
It shows that the cases with $t>2$ in [loc.cit.] in fact do not arise.
We will sketch alternative computer-free approaches to a proof of this part of the theorem
in Section~\ref{sec:2000}.

No novelty is claimed for the second part of Theorem~\ref{th:maintwo} (e.g.
it is already in \cite{MS96}), it is added for the sake of completeness.

\bigskip

Currently known (up to taking duals) examples of universal finite
polytopes of type $\{3,3,4,3,3\}$ are listed in Table~\ref{tab:known}.
We also give there information on
numbers of vertices $v$, facets $f$, groups, and whether there exists
$\Gamma\in\langle\{3,3,4,3\}_{\mathbf{s}},\{3,4,3,3\}_{\mathbf{t}}\rangle$
with group embedded as a subgroup in a ``$Y$-overgroup``, i.e. in a quotient
of $Y_{332}$ over relations affording $\phi$, as shown on Figure~\ref{fig:Yabc}(b).
The last example in the table is new, the rest already appear in \cite{MS96}.

\begin{table}[h!]
\begin{tabular}{ccrrcc}
\hline
$\mathbf{s}$ & $\mathbf{t}$ & $v$ & $f$ & $G$ & ``$Y$-overgroup''\\
\hline
(2000) & (2000) & $2^5$    & $2^5$       & $[2^{15} 3^2]$ & $O_8^-(2){:}2$\\
(2000) & (2200) & $2^5$    & $2^7$       & $[2^{18} 3^2]$ & $2^8{:}O_8^-(2){:}2$\\
(2200) & (2200) & $2^{11}$ & $2^{11}$    & $[2^{24} 3^2]$ & ---\\
(3000) & (3000) & \olp     &\olp         & $(3^2\times L_4(3)).2^2$ & ---\\
(2200) & (3000) & $\indinfi$ &
    $2^3{\cdot}3^5{\cdot} 5^2{\cdot} 7$ & $S_3\times \orst$ & $\Fi$ \\
\hline\\
\end{tabular}
\caption{The known universal finite polytopes
$\{\{3,3,4,3\}_{\mathbf{s}},\{3,4,3,3\}_{\mathbf{t}}\}$, their groups, and
(whenever they exist) ``$Y$-overgroups``.
\label{tab:known}}
\end{table}

Note that 
Lemma~\ref{lem:more33433} gives a construction of a finite polytope (not known to be
universal at present) with $\mathbf{s}=(6600)$,
$\mathbf{t}=(3000)$ as the {\em mix}, of the last 
two entries in Table~\ref{tab:known}.
The {\em mix} construction from \cite[Sect.~7A]{MS02}
provides a useful technical tool for constructing covers, as well as 2-covers,
of abstract polytopes, and, more generally, quotients of Coxeter complexes.

\begin{prop}\label{prop:prod}
(Mix construction \cite{MS02}). Let $G=\langle g_1,\dots,g_r\rangle$ and $H=\langle h_1,\dots,h_r\rangle$ be two quotients of
an arbitrary Coxeter group $F$ with $r$ generators $\gamma_k$, so that $g_k$ and $h_k$ are images
of $\gamma_k$ under the respective homomorphisms, $1\leq k\leq r$.
Let $G\diamond H:=\langle g_1 h_1, g_2 h_2,\dots g_r h_r\rangle\leq G\times H$. Then $G\diamond H$,
called \emph{mix} of $G$ and $H$, is
a quotient of $F$, with $g_k h_k$ the image of $\gamma_k$, for $1\leq k\leq r$. \qed
\end{prop}
Note that in general $G\diamond H$ might fail the intersection property, despite
it holding in $G$ and $H$, cf. \cite[Sect.~7A6]{MS02}.
Thus we will need extra tools to demonstrate the intersection property in $G\diamond H$, in
particular \cite[Lemma~12D4]{MS02}, which is Lemma~\ref{lem:ip} below.

Proposition~\ref{prop:prod} allows to mix two abstract polytopes with the same
diagram $\mathcal{D}$, or, more generally, with the groups affording the Coxeter relations
from $\mathcal{D}$ (in the latter case the actual Coxeter diagram might be a ``quotient'', in the same
sense as Weyl group $A_2\times A_2$ admits the relations of $F_4$).
The corresponding algebraic system is a meet-semilattice, corresponding
to the partially ordered set of certain normal subgroups in the Coxeter group
with diagram $\mathcal{D}$. One application of this is to construct covers of abstract polytopes,
used in Section~\ref{sec:gamma} to analyse the group of $\Gamma$ in Theorem~\ref{th:mainmain}.
For the sake of completeness, in Section~\ref{sec:prod}, cf. Proposition~\ref{prop:L43},
 we show how to use it to identify
the group of the universal polytope $\{\{3,3,4,3\}_{(3000)},\{3,4,3,3\}_{(3000)}\}$ with
$(3^2\times L_4(3)).2^2$, the example studied in a
great detail in \cite[Sect.4.2]{MR2419764}.

Another application of Proposition~\ref{prop:prod} is to construct certain 2-covers of abstract polytopes.
For instance, one constructs a finite abstract polytope of type \\
$\{\{3,3,4,3\}_\mathbf{s},\{3,4,3,3\}_\mathbf{t}\}$
for a new tuple $(\mathbf{s},\mathbf{t})$ of parameters, mixing a pair of examples
from Table~\ref{tab:known}.
More details on this are given in Section~\ref{sec:prod}.

\section{Notation and preliminaries}
Our notation for Coxeter groups and diagrams is standard. A Coxeter group $G$ is generated by
generators $g_1,\dots,g_r$, with the relations $g_i^2=(g_i g_j)^{k_{ij}}=1$,
for $1\leq i<j\leq r$ and $2\leq k_{ij}<\infty$ (more generally, one may assume
$k_{ij}=\infty$, but we will not use this here). Graphically this is drawn as an $r$-node
graph (diagram) $\mathcal{D}(G)$
with $k_{ij}-2$ edges joining vertices $i$ and $j$ (or, more generally, with
labels $k_{ij}$ on the edges for which $k_{ij}\geq 5$, but we will not use this here).
Groups $G$ for which $\mathcal{D}(G)$ is connected are called \emph{irreducible};
it is a necessary (although not sufficient) condition for
the irreducibility of the natural reflection representation of $G$ with generators $g_i$
being reflections in $\mathbb{R}^r$ with respect to hyperplanes $H_i$, with angles
between $H_i$ and $H_j$ determined by $k_{ij}$.
Subgroups $P_I:=\langle g_i\mid i \in I\rangle\leq G$ for $I\subseteq\{1,\dots,r\}$
are called \emph{parabolic} (called \emph{special parabolics} in \cite{MS02}).

For instance, the diagrams $Y_{\alpha,\beta,0}$,
$Y_{\alpha,1,1}$, $Y_{221}$, $Y_{321}$, $Y_{421}$ (cf.
Fig.~\ref{fig:Yabc}) correspond to irreducible finite Coxeter groups
$A_{\alpha+\beta}$, $D_{\alpha+2}$, $E_6$, $E_7$, and $E_8$,
which are also the Weyl groups of the respective root systems so named.
Note that the group $A_n$ is isomorphic to the symmetric group $S_{n+1}$.
If $\mathcal{D}(G)$ is an $r$-path, maybe with multiple edges, one talks about a \emph{string}
diagram, and encodes it as \emph{\schl}
$[k_{12},k_{23},\dots,k_{r-1,r}]$ (assuming the $g_i$ are ordered consecutively
on the path $\mathcal{D}(G))$. We have already mentioned an example of a string diagram $A_n$;
its \schl\ is $[3^{n-1}]$, where we abbreviated $n-1$ consecutive $3$'s
as $3^{n-1}$. Other examples we need here
of irreducible finite Coxeter groups with string diagram are $F_4$, with \schl\ $[3,4,3]$, cf.
Fig. \ref{fig:Fabc}, and $B_n$, with \schl\ $[4,3^{n-2}]$.

Further, we are interested in quotients of $G$ that
satisfy \emph{intersection property}, that is
\begin{equation}\label{eq:ip}
G_I\cap G_J= G_{I\cap J}\qquad\text{for all }I,J\subseteq\{1,\dots,r\}.
\end{equation}
For instance, finite quotients $p^6{:}E_6$, of the affine Coxeter group $\tilde{E}_6$,
i.e. the Coxeter group with diagram $Y_{222}$, and finite quotients $p^4{:}F_4$ of the
affine Coxeter group $\tilde{F}_4$, (the group with \schl\ $[3,3,4,3]$) satisfy the intersection
property whenever $p\geq 2$.

Quotients of $G$ with  string diagram that satisfy \eqref{eq:ip}
are called \emph{string $C$-groups} in \cite{MS96,MS02},
and they correspond to \emph{abstract $r$-polytopes}. The order complexes of these polytopes
are quotients of the
\emph{Coxeter complex} \cite[Chapter~2]{Ti:bldgs}, \cite{Pas:book,BC:book} of $G$,
a simplicial complex built from the cosets
of the parabolic subgroups $G_I\leq G$, $I\subseteq\{1,\dots,r\}$,
with respect to the respective normal subgroups (the latter might be trivial), cf.
\cite[Sect.~2C, 3D]{MS02}.
In particular for finite irreducible $G$
they are classical polytopes: $n$-simplices for $A_n$, $n$-cubes for $B_n$, the \emph{$24$-cell}
$\{3,4,3\}$ for $F_4$, etc. Abstract regular polytopes are also known as
satisfying the intersection property thin diagram geometries
with string diagram, see e.g. \cite{Pas:book,BC:book}.
By $\langle \mathcal{P}_f,\mathcal{P}_v\rangle$ we denote the class of all polytopes
with facets isomorphic to $\mathcal{P}_f$ and vertex figures isomorphic to $\mathcal{P}_v$.
This class contains unique polytope, called the \emph{universal polytope}
$\{\mathcal{P}_f,\mathcal{P}_v\}$, which covers each member of the class, cf. \cite[Thm.4A2]{MS02}.

In the sequel we will need a classification of abstract polytopes of type $\{3,3,4,3\}$
and, dually, $\{3,4,3,3\}$. Due to duality, it suffices to deal with the former, corresponding
to the quotients of the affine Coxeter group $\tilde{F}_4$, see Fig.~\ref{fig:eF4}. (In our notation
for \emph{conjugation}, $x^y=y^{-1}xy$.)
\begin{figure}[h]
\begin{tikzpicture}[every node/.style={circle,draw,inner sep=2pt}]
\node [draw] (c) [label=$a$] {};
\node [draw] (b) [label=$b$,right of=c] {} edge (c);
\node [draw] (a) [label=$c$,right of=b] {} edge (b);
\node [draw] (u) [label=$d$,right of=a] {} edge[double,thick] (a);
\node [draw] (s) [label=$e$, right of=u] {} edge (u);
\end{tikzpicture}\qquad $\sigma:=d^{cb}$, \quad $\tau:=c^{de}$
\caption{Affine Coxeter group $\tilde{F}_{4}$, nodes labelled by generators. \label{fig:eF4}}
\end{figure}
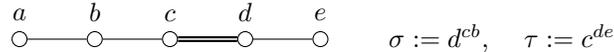

$F_4$ acts on the lattice
$\mathbb{Z}^4\cup ((\frac{1}{2},\frac{1}{2},\frac{1}{2},\frac{1}{2})+\mathbb{Z}^4$
of $\mathbb{R}^4$-vectors with coordinates being integer or half-integer,
cf. e.g. \cite[p.23]{MS96},
and one needs to classify the invariant sublattices of finite index. The latter
may be generated (as a submodule)
by a single vector, of the shape $\mathbf{s}:=(s,0,0,0)$, or $\mathbf{s'}:=(s,s,0,0)$, with
$2\leq s\in\mathbb{Z}$. Combining \cite[Theorems 3.4 and 3.5]{MS96}, one obtains the following.
\begin{theorem}\label{th:rank5}\cite{MS96}
Let $\Pi$ be a finite abstract regular polytope $\{3,3,4,3\}_\mathbf{t}$.
Then for $\mathbf{t}=\mathbf{s}$
its group is isomorphic to $[3,3,4,3]$ subject to the extra relation $(a\sigma\tau\sigma)^s=1$,
with $\tau$ and $\sigma$ as on Fig.~\ref{fig:eF4}, whereas for
$\mathbf{t}=\mathbf{s'}$ the extra relation is $(a\sigma\tau)^{2s}=1$.
\qed
\end{theorem}
In what follows we denote the groups arising in Theorem~\ref{th:rank5}
by $[3,3,4,3]_\mathbf{t}$ (or, if appropriate, by $[3,4,3,3]_\mathbf{t}$).

The mix construction in Proposition~\ref{prop:prod}
needs extra tools to show that intersection property holds in our mixes.
One of these is \cite[Lemma~12D4]{MS02} and some of its corollaries proved
immediately after [loc.cit.] which we state here in our context.
\begin{lemma}\label{lem:ip}
Let $H=[3,3,4,3]_{(s,0,0,0)}$ with $s$ an odd prime. Then the $[3,4,3]$-parabolic is maximal in $H$.
Consequently, let $G$ be a finite quotient of $[3^2,4,3^2]$
such that the following holds for its parabolics: 
$G_{\{0,\dots,4\}}\cong [3,3,4,3]_{(s,0,0,0)}$ and
$G_{\{1,\dots,5\}}\cong [3,4,3,3]_{(t,0,0,0)}$ (or $[3,4,3,3]_{(t,t,0,0)}$),
$t\geq 2$, and so that 
$G_{\{0,\dots,4\}}$ and $G_{\{1,\dots,5\}}$ are proper subgroups of $G$.
Then \eqref{eq:ip} holds in $G$. 
\qed
\end{lemma}

\section{Sporadic $Y$-presentations and their twists}\label{sec:Y}
Our Theorem \ref{th:mainmain} will follow at once from the following.
\begin{theorem}\label{th:main}
Let $\Gamma$ be a locally toroid rank $6$ abstract regular polytope
with vertex figures of type $\{3,4,3,3\}_{\mathbf{s}}$ and
facets of type $\{3,3,4,3\}_{\mathbf{t}}$.
Let $\mathbf{s}={(3000)}$ and $\mathbf{t}=(2200)$.
Then $\Gamma$ is finite and isomorphic to $\Gamma_k$, for some 
$1\leq k\leq 4$, defined in and after Theorem~\ref{th:mainmain}.
In particular, $\Gamma$ is a quotient of $\Gamma_4$, and
the group of $\Gamma_k$ is isomorphic to $S_k\times \Omega$.
\end{theorem}

According to \cite{Atl},
the sporadic simple group $\Fi$ has a presentation $Y_{332}$ on Fig.~\ref{fig:Yabc}(b)
subject to extra relations $S=f_{12}=f_{21}=1$, where
\begin{equation}\label{eq:firel}
S=(ab_1c_1ab_2c_2ab_3c_3)^{10}, \quad
f_{ij}=(ab_ib_jb_kc_ic_jd_i)^9,\text{ where }\{i,j,k\}=\{1,2,3\}.
\end{equation}
The following observation was the starting point of this project.
\begin{lemma}\label{lem:relsunderphi}
Let $\phi$ be the automorphism of the diagram $Y_{332}$ from  Fig.~\ref{fig:Yabc}(b).
Then $S^\phi$ and $S$ generate the same normal subgroup, and
$f_{12}^\phi=f_{21}$, $f_{21}^\phi=f_{12}$.
\end{lemma}
\begin{proof}
Note that $S, S^\phi\in Y_{222}=\tilde{E}_6$, and it is straightforward to check that they
generate the same $E_6$-invariant sublattice in the 6-dimensional $E_6$-invariant
lattice. The rest of the statement is obvious.
\end{proof}

In view of Lemma~\ref{lem:relsunderphi}, $\Fi$ contains a quotient $\Theta=\langle c_3,b_3,a,b_1b_2,c_1c_2,d_1d_2\rangle$
of a Coxeter group with diagram $[3^2,4,3^2]$.
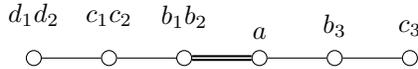
\begin{figure}[h]
\begin{tikzpicture}[every node/.style={circle,draw,inner sep=2pt}]
\node [draw] (c) [label=$d_1 d_2$] {};
\node [draw] (b) [label=$c_1 c_2$,right of=c] {} edge (c);
\node [draw] (a) [label=$b_1 b_2$, right of=b] {} edge (b);
\node [draw] (u) [label=$a$,right of=a] {}
    edge[double,thick] (a);
\node [draw] (s) [label=$b_3$, right of=u] {} edge (u);
\node [draw] (v) [label=$c_3$, right of=s] {} edge (s);
\end{tikzpicture}
\caption{Coxeter group with diagram $[3^2,4,3^2]$ from $\phi$.\label{fig:twistphi}}
\end{figure}

One immediately identifies
$\Theta_f=\langle b_3,a,b_1b_2,c_1c_2,d_1d_2\rangle <\Theta$
with a subgroup in $2^6{:}Sp_6(2)$ presented by
$Y_{331}$ with extra relations $f_{12}=f_{21}=1$,
and $\Theta_v=\langle c_3,b_3,a,b_1b_2,c_1c_2\rangle <\Theta$
with a subgroup in $3^5{:}E_6$ presented by $Y_{222}$ with an extra relation $S=1$, see
\cite[p.232-233]{Atl}.

Direct computations either by coset enumeration, or in
appropriate $E_6$- and $E_7$-invariant lattices allow one to identify $\Theta_v$
and $\Theta_f$ as groups classified by Theorem~\ref{th:rank5}.
By Lemma~\ref{lem:ip}, we have the following.
\begin{lemma}\label{lem:idF4}
The following isomorphisms hold: $\Theta_f=[3,4,3,3]_{(2200)}$, and
$\Theta_v=[3,3,4,3]_{(3000)}$. The intersection property holds in $\Theta$. \qed
\end{lemma}
At this point we can already conclude that the example of $\{3,3,4,3,3\}$-polytope
provided by $\Theta$ does not appear in \cite[Table~V]{MS96}, as
the pair of vectors $\mathbf{s}=(3000)$ and $\mathbf{t}=(2200)$ determining
the vertex figure and the facet is not present there.

As $\phi$ on Fig.~\ref{fig:Yabc}(b) induces an outer automorphism of $\Fi$
and centralises $\Theta$, one concludes by inspecting the list of maximal subgroups
of $\Fi$ that $\Theta\leq\orst$, and enumeration of cosets
of $\Theta$ in $\Fi$, or analysis of maximal subgroups of $\orst$ confirm that
in fact $\Theta=\orst$, identifying it with the group of the example $\Gamma_1$
in our Theorem~\ref{th:main}.

\begin{remark}\label{rem:Y}
The problem of recognising groups presented by $Y_{332}$ with similar extra relations,
generalising ones on \cite[pp.232-233]{Atl},
has been brought to the author's attention by Alexander A. Ivanov.
Namely, one is interested in the general setting where subgroups $Y_{331}=\tilde{E}_7$
and $Y_{222}=\tilde{E}_6$ are made
finite by factoring out the corresponding finite index invariant sublattices.

It is strikingly parallel to the problem of classifying the finite quotients
of $[3,3,4,3,3]$ we are concerned with in this text.
The simplest case, where $Y_{222}$ becomes $2^6{:}E_6$, has been essentially settled
in \cite{MR704292} (one gets orthogonal groups over $\mathbb{F}_2$, in some cases
acting on its natural $\mathbb{F}_2$-module), but in general it is open.
We remark that Proposition~\ref{prop:prod} allows to construct more examples, e.g. we can get
a subgroup of $\Fi\times 2^8{:}O_8^-(2)$ by applying it to the two known examples.

We hope to explore this topic further in another publication.
\end{remark}

\begin{problem}
The question of understanding similarly constructed subgroups of the
Monster-related groups with $Y$-presentations is largely open, although
they ought to exist by arguments similar to the $Y_{332}$-case considered here.

Preliminary experiments with enumerating cosets of one of our finite $[3^2,4,3^2]$-quotients,
namely the group of $\Gamma_1$, isomorphic to $\orst$,
in $[3^2,4,3^3]$, indicates that this gives the group $F_4(2)$, which is a subgroup
of a quotient of $Y_{333}$ isomorphic to $^2\, E_6(2)$.
More precisely, the enumeration returns index 12673024, which is 4 times the index
of $\orst$ in $F_4(2)$.

We should also mention that the $[3,4,3^3]$-subgroup in this group,
which corresponds to a locally toroidal polytope of the type
considered in  \cite[Sect.~12C]{MS02},
is isomorphic to $L_4(3){:}2$, shown to be a subgroup in $F_4(2)$
by L.~Soicher \cite{Soi:Y555more} by
an argument involving a diagram automorphism similar to $\phi$.

\end{problem}

\section{Another existence proof of $\Gamma_1$, and existence of its covers}\label{sec:gamma}
Here we present another, direct, existence proof of $\Gamma_1$, and
complete the proof of Theorem~\ref{th:main} by constructing its covers.

We begin by writing down a presentation $\mathcal{P}$ for $\Theta$, the group of
the \emph{universal cover}
of the examples from Theorem~\ref{th:main}.
In our context it simply means that in $\mathcal{P}$ we do not impose any
relations that involve all the generators. As $\Theta$ is the group of the (unique)
universal polytope, we can conclude without loss in generality that $\mathcal{P}$ is
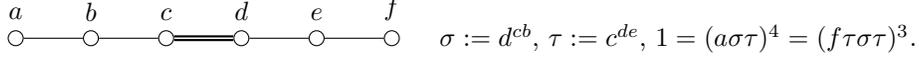
\begin{figure}[h]
\begin{tikzpicture}[every node/.style={circle,draw,inner sep=2pt}]
\node [draw] (c) [label=$a$] {};
\node [draw] (b) [label=$b$,right of=c] {} edge (c);
\node [draw] (a) [label=$c$,right of=b] {} edge (b);
\node [draw] (u) [label=$d$,right of=a] {} edge[double,thick] (a);
\node [draw] (s) [label=$e$, right of=u] {} edge (u);
\node [draw] (t) [label=$f$, right of=s] {} edge (s);
\end{tikzpicture}\quad
$\sigma:=d^{cb}$, $\tau:=c^{de}$, $1=(a\sigma\tau)^4=(f\tau\sigma\tau)^3$.
\caption{The presentation $\mathcal{P}$ for $\Theta$,  nodes labelled by generators. \label{fig:eeF4}}
\end{figure}

\begin{lemma}\label{lem:embeddingO}
There is a surjective homomorphism $\hat{ }:\Theta\to\orst$, with
its image, the group of $\Gamma_1$, satisfying the intersection property.
\end{lemma}

\begin{proof}
The embedding will be given by matrices in $\mathit{GL}_{24}(2)$,
with all the elements, except images $\hat{a}$ and $\hat{f}$ of $a$ and $f$, permutation matrices.
Namely, $\langle b,c,d,e\rangle=F_4$ embeds as the permutation group in its
natural action on the vertices of the 24-cell.
The latter action has an  imprimitivity system with three blocks of size 8, which
will correspond to the three natural 8-dimensional submodules for $O_8^+(2){:}2$.
We choose the generators so that the blocks are
$\{1..8\}$, $\{9..16\}$, and $\{17..24\}$, and so that the centre
of $F_4$ is the permutation $(1,2)(3,4)\dots (23,24)$. Let us slightly abuse notation and
view the following as specifications of
$24\times 24$-permutation matrices over $\mathbb{F}_2$.
\begin{align*}
\hat{b}=&\prod_{i\in\{0,8,16\}} (i+3,i+8)(i+4,i+7), &
\hat{c}=&(3,6)(4,5)(11,14)(12,13)(17,24)(18,23),\\
\hat{d}=&(5,6)\prod_{i=9}^{16} (i,i+8),&
\hat{e}=&(21,22)\prod_{i=1}^8 (i,i+8).
\end{align*}
In terms of the 24-cell, the size 8 imprimitivity block of a vertex $v$ consists of the
antipodal to $v$ vertex of the 24-cell, and the vertices antipodal to $v$ in
each of the 6 octahedra (facets of the 24-cell) on $v$. The
stabiliser of the block in $F_4$ is $\langle \hat{b},\hat{c},\hat{d},\hat{\tau}\rangle$.

Define matrices over $\mathbb{F}_2$, using $I_k$ to denote $k\times k$ identity matrix,
and $E_{ij}$ to denote the matrix with $1$ in the position $ij$ and $0$ elsewhere.
\begin{align*}
A&:=\begin{pmatrix}
1& 0& 0& 0& 0& 0& 0& 0\\
0& 1& 1& 1& 1& 1& 0& 0\\
1& 0& 1& 0& 0& 0& 0& 1\\
1& 0& 0& 1& 0& 0& 0& 1\\
1& 0& 0& 0& 1& 0& 0& 1\\
1& 0& 0& 0& 0& 1& 0& 1\\
0& 0& 1& 1& 1& 1& 1& 0\\
0& 0& 0& 0& 0& 0& 0& 1
\end{pmatrix}, &
F&:=\begin{pmatrix}
0& 0& 0& 0& 1& 1& 0& 0\\
0& 0& 0& 0& 1& 0& 0& 0\\
0& 0& 0& 0& 0& 0& 1& 1\\
0& 0& 0& 0& 0& 0& 1& 0\\
0& 1& 0& 0& 0& 0& 0& 0\\
1& 1& 0& 0& 0& 0& 0& 0\\
0& 0& 0& 1& 0& 0& 0& 0\\
0& 0& 1& 1& 0& 0& 0& 0
\end{pmatrix},\\
\hat{a}&:=I_3\otimes A, &
\hat{f}&:=\begin{pmatrix} I_8+E_{21}&0&0\\0&0&F\\0&F&0 \end{pmatrix}.
\end{align*}
Now it is routine to check that the relations in $\mathcal{P}$ hold for just defined $\hat{a}$,\dots, $\hat{f}$. 
As well, the elements $\hat{a}$,\dots, $\hat{f}$ 
leave invariant the quadratic form $\Phi:=\sum_{j=1}^{24}x_j^2+\sum_{i=1}^{12}x_{2i-1}x_{2i}$; this is immediate
for the permutation matrices $\hat{b}$,\dots, $\hat{e}$, as they commute with the centre of $\langle \hat{b},\hat{c},\hat{d},\hat{e}\rangle$,
and a direct computation for $\hat{a}$ and $\hat{f}$.

As this point we have $\hat{\Theta}\leq \orst$.
To see that the latter  is in fact equality, one can either inspect the list of the maximal subgroups of $\orst$,
or investigate the action of the stabiliser $H:=\langle \hat{a},\hat{b},\hat{c},\hat{d},\hat{\tau},\hat{f}\rangle$
of the subspace spanned by the first 8 coordinates in $\hat{\Theta}$,
and see that it acts transitively on  the 135 nonzero $\Phi$-isotropic vectors there,
by computing the orbit of the all-1 vector, and further identifying $H$ with a rank 3 permutation group $O_8^+(2){:}2$
in its natural action on isotropic vectors.

Finally, the intersection property holds, by \cite{MS96}.
\end{proof}

Let $\psi:\Theta\to S_4$ be a map defined on generators of $\Theta$, with $\psi(a):=(1,2)$, $\psi(b):=(2,3)$, $\psi(c):=(3,4)$,
and $\psi(d)=\psi(e)=\psi(f)=()$. It turns out that it extends to a homomorphism onto. (Note that there is no symmetry here: if we try
to send $a$, $b$, $c$ to the identity instead, we get a trivial group).

\begin{lemma}\label{lem:coversO}
The map $\psi$ is a homomorphism onto. Also,
there is a homomorphism from $\Theta$ onto $S_4\times \orst$, with
its image the group of $\Gamma_4$.
\end{lemma}
\begin{proof}
For the first part, observe that the Coxeter relations of $\Theta$ hold in $\psi(\Theta)$. Also,
$\psi(\sigma)=()$ and $\psi(\tau)=\psi(c)=(3,4)$. It follows that the rest of the relations hold, as well,
and we are done.

For the second part,
we construct an embedding of $\Theta$ into $G:=S_4\times \orst=\psi(\Theta)\times \hat{\Theta}$,
as prescribed by Proposition~\ref{prop:prod}.
Consider $H:=\psi(\Theta)\diamond\Theta\leq G$; recall that it is
generated by $\psi(a)\hat{a}$,\dots, $\psi(f)\hat{f}$.
Trivially, by checking the relations, $H$ is a homomorphic image of $\Theta$.
Moreover, the intersection property holds in $H$ by Lemma~\ref{lem:ip}. 

It remains to show $H=G$. Consider an element of $O_8^+(2)$ of order 5,
as a word $\hat{w}:=w(\hat{a},\dots,\hat{f})$ in our generators, and the
corresponding word $w:=w(\psi(a)\hat{a},\dots,\psi(f)\hat{f})\in H$.
Then $w=w(\psi(a),\dots,\psi(f))\hat{w}$. Thus $w^{12}=\hat{w}^2\in\hat{\Theta}$ is of order 5.
As $O:=O_8^+(2)$ is generated by such elements, and as $\hat{\Theta}$ is
generated by $O$ and $\hat{d}$, $\hat{e}$, $\hat{f}$, we have $\hat{\Theta}\leq H$.
Therefore $H=G$, as claimed.
\end{proof}

\section{The case $\mathbf{s}={(2000)}$}\label{sec:2000}
First, we prove Theorem~\ref{th:maintwo}.
Let $\Gamma=\{\{3,3,4,3\}_{(2000)},\{3,4,3,3\}\}$. Then its group $\Theta$ has presentation
as follows.
\begin{figure}[h]
\begin{tikzpicture}[every node/.style={circle,draw,inner sep=2pt}]
\node [draw] (c) [label=$a$] {};
\node [draw] (b) [label=$b$,right of=c] {} edge (c);
\node [draw] (a) [label=$c$,right of=b] {} edge (b);
\node [draw] (u) [label=$d$,right of=a] {} edge[double,thick] (a);
\node [draw] (s) [label=$e$, right of=u] {} edge (u);
\node [draw] (t) [label=$f$, right of=s] {} edge (s);
\end{tikzpicture}\quad
$\sigma:=d^{cb}$, $\tau:=c^{de}$, $1=(a\sigma\tau\sigma)^2$.
\end{figure}

Enumeration of cosets of the facet stabiliser $\langle a,b,c,d,e\rangle$ in $\Theta$
shows that $|\Theta|=2^{18}{\cdot} 3^2$. Working in the
resulting permutation representation one computes the orders of
$f\tau\sigma\tau$ and $f\tau\sigma\tau\sigma$ to be 4, thus
identifying the stabiliser of the vertex with the one corresponding to
the type $\{3,4,3,3\}_{(2200)}$, as claimed.

The polytope $\Gamma$ corresponds to
the second entry in \cite[Table 12D1]{MS02} (or \cite[Table~V]{MS96}) for $t=2$.
The quotient by a normal subgroup $2^3$ in $\Theta$ produces
the other case, $\mathbf{t}=(2000)$, the first entry in [loc.cit.] for $t=2$.
It shows that the cases with $t>2$ in [loc.cit.] in fact do not arise.
\qed

\medskip

More conceptually, one can notice that it is proved on \cite[page 463]{MS02} that
the number of vertices in $\Gamma$ equals 32. In particular, it means that the index
of the subgroup $H:=\langle b,c,d,e,f\rangle$ in $\Theta$ is 32.
Let $\pi:\Theta\to\hat{\Theta}$ be the permutation representation of $\Theta$
on the cosets of $H$, and $K$ its kernel. As $H$ is a quotient of $\tilde{F}_4$,
and $K$ is a normal subgroup there, we see that $K$ is very small. It cannot contain
a faithful 4-dimensional module for $F_4$, as the latter would admit an action of $a$
commuting with the action of $\langle d,e,f\rangle$, but still inducing a
faithful action of $\langle a,b\rangle$ (as the action of $\langle b,c\rangle$ would be faithful,
and the latter is conjugate to $\langle a,b\rangle$).
This implies that $|K|\leq 8$.

Another possibility is to carry out the computation of
a presentation for $H$ applying the Reidemeister--Schreier algorithm
(which produces a presentation of a finite index subgroup of a finitely
presented group)
cf. e.g. \cite{MR1812024}, and find that $H=[3,4,3^2]_{(2200)}$.

\subsection{Embedding of the group of $\Gamma$ in $Y$-group.}\label{subs:emb}
Consider the 2-modular
quotient $G$ of $Y_{332}$, which is isomorphic to $G:=2^8{:}O_8^-(2){:}2$, as established by R.~Griess
in \cite[pp.277-278]{MR704292}. Then $\Theta$ can be identified with its subgroup,
as shown on Figure~\ref{fig:twistphi}, under the automorphism
$\phi$ of the $Y$-diagram on Figure~\ref{fig:Yabc}(b).
Specifically, the relation forcing the $\{3,3,4,3\}_{(2000)}$-facet turns the subgroup $Y_{222}$ into
$2^6{:}U_4(2).2$, and, respectively, $Y_{331}$ into $2^7{:}Sp_6(2)$.

Adding further relation forcing the $\{3,4,3,3\}_{(2000)}$-vertex figure kills
$O_2(G)$, and one obtains a quotient of
$\{\{3,3,4,3\}_{(2000)},\{3,4,3,3\}_{(2000)}\}$ by a subgroup of order 4.

The group $G$ also features on \cite[pp.232-233]{Atl} as a quotient of
$X_{332}$, a re-labelling of $Y_{332}$, with the node $c_3$ of the latter
denoted by $a_3$, by an explicitly given relation $W$. In complete analogy with
the $\Fi$ situation, the relation $W$ affords the automorphism $\phi$, and
allows one to establish the existence of $\Gamma$ directly.

\begin{remark}
In his PhD thesis L.~Soicher \cite[Thm.A.2]{Soi:phd} proved that the group $[3^m,4,3^n]$, subject to
a relation killing the centre of the $[3,4,3]$-parabolic, is
isomorphic to $2^{mn}{:}(S_{m+1}\times S_{n+1})$.
Note that for $m=n=2$ this example has the cover
$\{\{3,3,4,3\}_{(2000)},\{3,4,3,3\}_{(2000)}\}$ from the last part of
Theorem~\ref{th:maintwo} (see also the first entry in \cite[Table V]{MS96}).

Another ``tower'' of finite $[3^m,4,3^n]$-examples with unbounded $m$ and $n$ may be constructed
from 2-modular quotients of the group $Y_{\alpha\alpha\gamma}$ by twisting with $\phi$,
generalising the observation in the beginning of Section~\ref{subs:emb}.
There the $[3^2,4,3^2]$-subgroups will correspond to
$\{\{3,3,4,3\}_{(2000)},\{3,4,3,3\}_{(2200)}\}$ from Theorem~\ref{th:maintwo}.
\end{remark}

\section{The mix construction}\label{sec:prod}
Our first application of the mix construction in Proposition~\ref{prop:prod} was Lemma \ref{lem:coversO}.
Here we present more applications of the mix construction.
The following constitutes another proof of a result in \cite[Sect.4.2]{MR2419764}.
\begin{lemma}\label{prop:L43}
The group of $\Gamma=\{\{3,3,4,3\}_{(3000)},\{3,4,3,3\}_{(3000)}\}$
is isomorphic to $(3^2\times L_4(3)).2^2$.
\end{lemma}
\begin{proof} (Sketch)
From coset enumeration, we know the order of the group in question.
The group $\hat{\Theta}:=L_4(3){:}2_2\cong \mathit{PGO}_6^+(3)$ has a maximal subgroup isomorphic to
$[3,3,4,3]_{(3000)}$, according to \cite{Atl}. As well, its central extension by the group of 
order 2, which is in fact split, is constructed in
 \cite[Sect.4.2]{MR2419764} as a quotient of the group $\Theta$ of $\Gamma$ (more
precisely, as the reduction modulo 3 of the natural 6-dimensional
representation of the Coxeter group $[3,3,4,3,3]$).

Let $\Omega$ be the quotient of $\Theta$ obtained by imposed the extra relation that
the generators (in the Coxeter diagram)
of the ``middle'' dihedral subgroup of order 8 commute. It is easy to check that
$\Omega=S_3\times S_3$. Consider the mix $P:=\hat{\Theta}\diamond\Omega$.
One sees that the $P$ contains a copy of the simple group $L_4(3)$, by considering
words in $P$ corresponding to elements of $\hat{\Theta}$ of order 13. Thus it also
contains $O_3(\Omega)=3^2$. It remains to observe that the Abelian invariants of $\Omega$
and of $[3^2,4,3^2]$ are equal to $(2,2)$, implying that the $2^2$ acting on $P$ is generated
by $P$. Finally, $P$ satisfies the intersection property by Lemma~\ref{lem:ip}.
\end{proof}

\subsection{Constructing 2-covers}
The following
will be needed for rank 6 cases.
\begin{prop}{\label{prop:prod3343}}
Let $\mathbf{s}=(s,s',0,0)\in \{(s,s,0,0), (s,0,0,0)\}$ and
$\mathbf{t}=(t,t',0,0)\in \{(t,t,0,0), (t,0,0,0)\}$.
Denote $\ell:=\lcm(s,t)$. The mix $G\diamond H$ of
$G=[3,3,4,3]_{\mathbf{s}}$ and $H=[3,3,4,3]_{\mathbf{t}}$
is $[3,3,4,3]_{(\ell,\ell,0,0)}$ in the following cases
\begin{enumerate}
\item $s'=s$, $t'=t$;
\item $s'=s$, $t'=0$, $2\ell=\lcm(2s,t)$.
\item $s'=0$, $t'=t$, $2\ell=\lcm(s,2t)$.
\end{enumerate}
Otherwise $G\diamond H$ is $[3,3,4,3]_{(\ell,0,0,0)}$.
\end{prop}
\begin{proof}
To establish the intersection property \eqref{eq:ip}, we embed $G\diamond H$ ``diagonally''
into $G\times H$ considered as an affine matrix group over the suitable
finite ring. This shows that $G\diamond H$ is isomorphic to some
$[3,3,4,3]_{\mathbf{u}}$ and so \eqref{eq:ip} holds. It remains to compute the relation $\mathbf{u}$
implied by $\mathbf{s}$ and $\mathbf{t}$.

Let $\mathbf{s}$ and $\mathbf{t}$ have the same ``format'', i.e. either $s'=t'=0$ or
$s'=s$, $t'=t$. Then the only difference between the presentations of $G$ and $H$ is the
exponent of the extra non-Coxeter relation in Theorem~\ref{th:rank5}, and the claim follows.

It remains to deal with the case of different ``formats''.
Let $G=\langle a,b,c,d\rangle$, as in Theorem~\ref{th:rank5}.
Respectively, let $H=\langle a',b',c',d'\rangle$, and denote
 $\sigma':=d'^{c'b'}$, $\tau':=e'^{c'd'}$.

Let $s'=s$ and $t'=0$.
Thus the orders of $a\sigma\tau$ and $a\sigma\tau\sigma$ equal to $2s$,
while the orders of $a'\sigma'\tau'$ and $a'\sigma'\tau'\sigma'$ equal $2t$ and
$t$, respectively. Then the order of $aa'\sigma\sigma'\tau\tau'$ equals
$2\ell$, while the order of $aa'\sigma\sigma'\tau\tau'\sigma\sigma'$ equals
$\lcm(2s,t)$. We get the case $[3,3,4,3]_{(\ell,\ell,0,0)}$
if and only if the two latter orders are equal.
The case of $s'=0$, $t'=t$ is dealt with in the same way.
\end{proof}

This immediately implies
\begin{lemma}\label{lem:more33433}
There exist $\{\{3343\}_{\mathbf{s}},\{3433\}_{\mathbf{t}}\}$ for
$(\mathbf{s},\mathbf{t})=(3000,6600)$.
\end{lemma}
\begin{proof}
Apply the mix construction to
$\{\{3343\}_{(3000)},\{3433\}_{(2200)}\}$ and \\
$\{\{3343\}_{(3000)},\{3433\}_{(3000)}\}$, compute
$\mathbf{s}$ and $\mathbf{t}$ using Proposition~\ref{prop:prod3343}, and note that the intersection property holds
by Lemma~\ref{lem:ip}.
\end{proof}

\begin{remark}
One can also consider other mixes of examples in
Table~\ref{tab:known} to obtain, e.g.
$(\mathbf{s},\mathbf{t})=(6000,2200)$; however,
we do not know whether the intersection property holds
in this group and other similar examples,
as Lemma~\ref{lem:ip} would not apply.
\end{remark}

\begin{remark}
Analogously one can try to construct new examples of \\
$\{\{3,4,3,3\}_{\mathbf{s}},\{4,3,3,4\}_{\mathbf{t}}\}$.
Here one needs to be able to compute in the meet-semilattice
generated by the parameters $\mathbf{t}$ of the groups $[4,3,3,4]_{\mathbf{t}}$, for
$\mathbf{t}$ in the corresponding column of \cite[Table VI]{MS96} and in the examples
from \cite{MR2610280}. We leave this to another publication.
\end{remark}

\subsection*{Acknowledgements}
The author thanks Alexander A. Ivanov for bringing the topic of $Y_{332}$-presentations, cf. Remark~\ref{rem:Y},
to his attention, Daniel Allcock, Leonard Soicher for useful suggestions and remarks, and an
anonymous referee for catching errors---in particular misuse
of the mix construction---in a previous version of this text.
The author is supported in part by
the EU Horizon 2020 research and innovation programme, grant agreement
OpenDreamKit No 676541.

\bibliography{geom}
\bibliographystyle{abbrv}
\end{document}